\documentclass[12pt]{article}
\usepackage[margin=1.25in]{geometry}
\usepackage[utf8]{inputenc}

\usepackage[english]{babel}
\usepackage{natbib}
\setcitestyle{authoryear,open={(},close={)}} 

\usepackage[dvipsnames]{xcolor}

\usepackage{amsmath, amssymb, amsthm,mathtools,dsfont}
\usepackage{bbm}
\usepackage{natbib}
\usepackage{blkarray}
\usepackage{cancel}
\usepackage{enumitem}
\usepackage{euscript}
\usepackage{float}
\usepackage{bm}
\usepackage{graphicx}
\usepackage{mathrsfs}
\usepackage{microtype}
\usepackage{ragged2e}
\usepackage{sectsty}
\usepackage{thmtools}
\usepackage{tikz}
\usepackage[Symbolsmallscale]{upgreek}

\usepackage{microtype}
\usepackage{graphicx}
\usepackage{subfigure}
\usepackage{hyperref}

\usepackage[normalem]{ulem}

\hypersetup{citecolor=purple, colorlinks=true, linkcolor=blue}

\newcommand{\cH}{\mathcal{H}}

\newcommand{\cP}{\mathcal{P}}

\newcommand{\cT}{\mathcal{T}}

\newcommand{\cX}{\mathcal{X}}
\newcommand{\cY}{\mathcal{Y}}

\newcommand{\RR}{\mathbb{R}}

\newcommand{\kl}[2]{\text{KL}(#1 \| #2)}
\newcommand*{\tv}[2]{\mathrm{d_{TV}}(#1, #2)}

\newcommand*{\triplenorm}[1]{{\left\vert\kern-0.25ex\left\vert\kern-0.25ex\left\vert #1
    \right\vert\kern-0.25ex\right\vert\kern-0.25ex\right\vert}}

\DeclareMathOperator{\dom}{dom}

\newcommand{\R}{\mathbb{R}}
\newcommand{\Rd}{\mathbb{R}^d}
\renewcommand{\phi}{\varphi}
\newcommand{\eps}{\varepsilon}

\newcommand*{\E}{\mathbb E}

\newcommand*{\ep}{\varepsilon}
\newcommand*{\defeq}{\coloneqq}
\newcommand*{\rd}{\mathrm{d}}
\newcommand*{\dd}{\, \rd}

\newcommand{\OTep}{W_{2,\eps}^2}

\newcommand{\p}[1]{\left(#1 \right)}

\newcommand{\Hphi}{H_{\max}(\phi_\eps^\nu)}
\newcommand{\Hpsi}{H_{\max}(\psi_\eps^\nu)}
\usepackage{fancyhdr}

\usepackage[capitalize,noabbrev]{cleveref}

\theoremstyle{plain}
\newtheorem{theorem}{Theorem}[section]

\newtheorem{prop}[theorem]{Proposition}
\newtheorem{lemma}[theorem]{Lemma}
\newtheorem{corollary}[theorem]{Corollary}
\theoremstyle{definition}

\theoremstyle{remark}
\newtheorem{remark}[theorem]{Remark}

\title{Tight stability bounds for entropic Brenier maps}
\author{Vincent Divol\thanks{CEREMADE, Université Paris-Dauphine - PSL. \tt vincent.divol@psl.eu} \ \ Jonathan Niles-Weed\thanks{Courant Institute of Mathematical Sciences and Center for Data Science, New York University. \tt jnw@cims.nyu.edu} \ \  Aram-Alexandre Pooladian\thanks{Center for Data Science, New York University. \tt aram-alexandre.pooladian@nyu.edu}}
\begin{document}

\maketitle

\begin{abstract}
Entropic Brenier maps are regularized analogues of Brenier maps (optimal transport maps) which converge to Brenier maps as the regularization parameter shrinks. In this work, we prove quantitative stability bounds between entropic Brenier maps under variations of the target measure. In particular, when all measures have bounded support, we establish the optimal Lipschitz constant for the mapping from probability measures to entropic Brenier maps. This provides an exponential improvement to a result of Carlier, Chizat, and Laborde (2024). As an application, we prove near-optimal bounds for the stability of semi-discrete \emph{unregularized} Brenier maps for a family of discrete target measures.
\end{abstract}

\section{Introduction}\label{sec: intro}
The theory of optimal transport defines a geometry over probability measures via the \emph{$2$-Wasserstein distance}: for a source measure $\rho$ and a target measure $\mu$ with finite second moments, their Wasserstein distance is given by 
\begin{align}\label{eq:w22_intro}
     W_2^2(\rho,\mu) \defeq \min_{T  :  T_\sharp \rho = \mu} \int \|x - T(x)\|^2 \dd \rho(x)\,,
\end{align}
where the constraint $T_\sharp \rho = \mu$ means that for $X \sim \rho$, $T(X) \sim \mu$, i.e., $T$ is a transport map. The minimizer to \eqref{eq:w22_intro}, when it exists, is called an optimal transport map, which we denote by $T_0^\mu$. A seminal result by \citet{Bre91} states that a unique optimal transport map between $\rho$ and $\mu$ exists whenever $\rho$ has a density, and moreover $T_0^\mu = \nabla \phi_0^\mu$, where $\phi_0^\mu$ is some convex function. We will henceforth refer to optimal transport maps as \emph{Brenier maps}, and the corresponding convex functions that generate them as \emph{Brenier potentials}.

A long-standing question in the optimal transport community is the following: is the mapping $\mu\mapsto T_0^\mu$ H{\"o}lder continuous with respect to the $2$-Wasserstein distance? In other words, do there exist constants $C,\beta > 0$ such that for all probability measures  $\mu$, $\nu$ with finite second moments,
\begin{align}\label{eq:stability_brenier}
    \|T_0^\mu - T_0^\nu\|_{L^2(\rho)} \leq C W^\beta_2(\mu,\nu)\,?
\end{align}
Since the inequality $W_2(\mu,\nu) \leq \|T_0^\mu - T_0^\nu\|_{L^2(\rho)}$ always holds, \eqref{eq:stability_brenier} would imply that the mapping $\mu\mapsto T_0^\mu$ is a bi-H{\"o}lder embedding of the Wasserstein space into $L^2(\rho)$.
We call such an inequality a \emph{stability bound}.

The unique structure of the one-dimensional optimal transport problem shows that when $\rho$, $\mu$, and $\nu$ are probability  measures on $\RR$, the bound \eqref{eq:stability_brenier} holds with $C = \beta=1$---that is, the mapping $\mu \mapsto T_0^\mu$ is an isometry~\citep[see, e.g.,][Chapter 2]{Panaretos2020}.
On the other hand, \citet{andoni2015snowflake} showed that if $d \geq 3$, then \eqref{eq:stability_brenier} cannot hold uniformly over all probability measures $\mu$ and $\nu$ on $\RR^d$ with finite second moment.
In fact, their main statement is significantly stronger and rules out the possibility of embedding the Wasserstein space into any $L^p$ space, even in a very weak sense.
Nevertheless, as we describe further below, a stability bound such as~\eqref{eq:stability_brenier} can hold if further conditions are imposed on $\mu$ and $\nu$, for instance, if they are compactly supported.

An early investigation in this direction is due to \citet{gigli2011holder}, who showed that even when $\mu$ and $\nu$ are compactly supported, the exponent in~\eqref{eq:stability_brenier} cannot be better than $\beta=\tfrac12$. However, in the same paper, the author reports a simple proof due to Ambrosio that shows that if one of the Brenier maps, say $T_0^\nu$, is $\Lambda$-Lipschitz, then $\beta=\tfrac12$ is achievable, with $C = 2\sqrt{\Lambda R}$, where $R$ is the diameter of the support of $\rho$; see also \citet[Theorem 2.3]{merigot2020quantitative} for a precise statement and proof of this result.
More recently, \cite{manole2021plugin} showed that if $T_0^\nu$ is $\Lambda$-Lipschitz and its inverse is $1/\lambda$-Lipschitz, then $\beta=1$ is achievable, with constant $C = \sqrt{\Lambda/\lambda}$.

Though these positive results are encouraging, requiring \emph{a priori} smoothness bounds on one of the two Brenier potentials excludes many cases of practical interest, for instance, the case of discontinuous Brenier maps.
Such maps arise commonly in applications of optimal transport to machine learning, where it is natural to consider probability measures that lie on a union of manifolds of different intrinsic dimension~\citep{brown2022union}.
There has therefore been significant recent interest in obtaining stability bounds without such assumptions; see \cite{merigot2020quantitative,berman2021convergence,delalande2021quantitative}. The results of \cite{delalande2021quantitative} are the most recent. They show that if $\rho$ has a (uniformly upper and lower bounded) density supported on a convex set $\cX$, with $\mu$ and $\nu$ also supported on a compact set $\cY$, then
\begin{align*}
    \| T_0^\mu - T_0^\nu\|_{L^2(\rho)} \leq C_{d,\cX, \cY,\rho} W^{1/6}_2(\mu,\nu)\,.
\end{align*}
In fact, the authors prove this bound for the $W_1$ distance. Their proof technique relies on applications of the Brascamp--Lieb and Prekopa--Leindler inequalities.

The goal of this paper is to prove analogous stability bounds for \emph{entropic Brenier maps}.
Entropic Brenier maps are defined as barycentric projections of entropic optimal couplings between $\rho$ and $\mu$ (resp. $\rho$ and $\nu$), written in short-hand as $T_\eps^{\mu}$ (resp. $T_\eps^{\nu}$), where $\eps>0$ is a regularization parameter (we defer detailed background to \cref{sec: background_eot}).
Our motivation for studying such maps is twofold.
First, entropic Brenier maps are prominent in applications for their statistical and computational properties.
These maps arise directly from the output of \emph{Sinkhorn's algorithm}, a very popular computational approximation to optimal transport~\citep{PeyCut19}.
Moreover, \cite{pooladian2021entropic} advocated for their use as a statistical tool, and analyzed the entropic Brenier map as an estimator for the unregularized Brenier map on the basis of i.i.d.~data from $\rho$ and $\mu$.
Second, entropic optimal transport is a natural smoothed analogue to the optimal transport problem, and it is reasonable to hope that techniques developed for entropic optimal transport can give insights into the structure of the unregularized problem.

Despite the importance of entropic Brenier maps, much less is known about their stability properties.
The first result in this area is due to \cite{carlier2022lipschitz}, who showed that if $\rho$, $\mu$, and $\nu$ are compactly supported, then
\begin{align}\label{eq:stability_entbrenier_carlier}
    \| T_\eps^{\mu} - T_\eps^{\nu}\|_{L^2(\rho)} \leq C_\eps W_2(\mu,\nu)\,,
\end{align}
where $C_\eps$ is a constant that grows exponentially as $\eps$ tends to zero.

This striking result reveals that entropic Brenier maps automatically enjoy better stability properties than unregularized Brenier maps when $\eps > 0$.
However, if~\eqref{eq:stability_entbrenier_carlier} is to be used to extract either practical bounds for entropic Brenier maps or insights about unregularized Brenier maps in the $\eps \to 0$ limit, it is crucial to obtain sharp bounds on the constant $C_\eps$.

\subsection*{Contributions}
The goal of this paper is to improve the Lipschitz constant for the embedding $\mu \mapsto T_\eps^\mu$ as a function of $\eps$. Our main theorem is technical, but it readily implies results in the following three scenarios of interest.

First, if the source and target measures are merely supported in the Euclidean ball of radius $R$, then
\begin{align*}
    \| T_\eps^{\mu} - T_\eps^{\nu}\|_{L^2(\rho)}  \leq \bigl(1 + 2R^2/\eps\bigr) W_2(\mu,\nu)\,,
\end{align*}
see \cref{cor:stab_1}. We stress that none of the measures here require densities, and so, a priori, Brenier maps may not exist, while their entropic counterparts do. Moreover, up to universal constants, we show that this bound is tight; see \Cref{rmk:tight}. This is an exponential improvement on the bounds provided by \cite{carlier2022lipschitz}.

As in the unregularized case, the preceding bounds can be improved under smoothness assumptions on the entropic Brenier potentials.
Such assumptions are arguably more reasonable than in the unregularized case, since it is sometimes possible to obtain \emph{a priori} smoothness bounds for entropic Brenier potentials via elementary tools \citep[see, e.g.,][]{chewi2023entropic}.
If one of the entropic Brenier maps, say $T_\eps^\nu$, is $\Lambda$-Lipschitz, we show that the previous bound can be improved to
\begin{align*}
    \| T_\eps^{\mu} - T_\eps^{\nu}\|_{L^2(\rho)} \leq \bigl(1 + 2\sqrt{{R\Lambda/\eps}}\bigr) W_2(\mu,\nu)\,,
\end{align*}
Going further, if the backward entropic Brenier map  $S_\eps^\nu$ (see \cref{sec: background_eot} for a precise definition) is $1/\lambda$-Lipschitz, then the bound becomes independent of the regularization parameter:
\begin{align*}
        \| T_\eps^{\mu} - T_\eps^{\nu}\|_{L^2(\rho)} \leq \bigl(1 + 2\sqrt{{\Lambda/\lambda}}\bigr) W_2(\mu,\nu)\,.
\end{align*}
See \cref{cor:stab_2} for these last two results. In particular, up to constants, this result is analogous to the stability bound established by \citet{manole2021plugin}.

As a novel application, we turn to the \emph{semi-discrete} setting of optimal transport, where $\mu$ and $\nu$ are both supported on finitely many atoms and $\rho$ has a sufficiently well-behaved density. In this setting, we partially close the gap left by Gigli and others, where we prove that
\begin{align}\label{eq:quantstab_intro}
    \| T_0^\mu - T_0^\nu\|_{L^2(\rho)} \lesssim  W^{1/3}_2(\mu,\nu)\,,
\end{align}
{where $\mu,\nu$ satisfy appropriate regularity conditions, as does the source measure $\rho$, and the suppressed constant depends on these regularity assumptions. While our results do not allow for arbitrary discrete measures, they hold for a wide class of discrete measures and do not require the support of the atoms to be the same. The proof starts from the following application of the triangle inequality
\begin{align*}
\| T_0^\mu - T_0^\nu\|_{L^2(\rho)} \leq \|T_0^\mu - T_\eps^\mu\|_{L^2(\rho)} + \|T_0^\nu - T_\eps^\nu\|_{L^2(\rho)} + C_\eps W_2(\mu,\nu)\,.
\end{align*}
Under appropriate assumptions on $\rho$ and the two discrete measures $\mu$ and $\nu$, we are able to control the first two terms using existing techniques, and the third term can be controlled via \cref{cor:stab_1}. Balancing the resulting terms as a function of $\eps$, we obtain the final bound that appears in \eqref{eq:quantstab_intro}.
Our identification of the sharp constant $C_\eps$ is crucial to obtaining the result.
See \cref{sec:quantstab_semidiscrete} for more details.

\section{Background on optimal transport}\label{sec: background_all}
\subsection{Optimal transport}\label{sec: background_ot}
Let $\cP_2$ be the space of probability measures on $\R^d$ with finite second moment.
For two probability measures $\rho,\nu \in \cP_2$, we define the \textit{(squared) $2$-Wasserstein distance} by
\begin{align}\label{eq:kant_ot_p}
\tfrac12W_2^2(\rho,\nu) := \inf_{\pi \in \Gamma(\rho,\nu)} \iint \tfrac12\|x - z\|^2 \dd \pi(x,z)\,,
\end{align}
where $ \Gamma(\rho,\nu)$ is the set of joint measures with left-marginal $\rho$ and right-marginal $\nu$; see \cite{Vil08, San15} for more information. When $\rho,\nu$ have finite second moments, solutions to \eqref{eq:kant_ot_p} always exist. When $\rho$ has a density, \eqref{eq:kant_ot_p} has a unique minimizer, $\pi_0$, called the the \emph{optimal transport coupling}.

From \eqref{eq:kant_ot_p}, one can derive the \emph{dual formulation} of the 2-Wasserstein distance
\begin{align}\label{eq:kant_ot_d_2}
 \tfrac12W_2^2(\rho,\nu) &= \tfrac12M_2(\rho + \nu) - \inf_{\phi \in L^1(\rho)} \int \phi \dd \rho + \int \phi^* \dd \nu\,,
\end{align}
where $M_2(P) :=  \int \|x\|^2 \dd P(x)$ for a measure $P$, and $\phi^*$ is the convex conjugate of $\phi$, defined for $x\in \R^d$ by
$\phi^*(x) \defeq \sup_y \{\langle x,y \rangle - \phi(y)\}.$

Similar to the primal problem, whenever $\rho$ has a density, the dual problem \eqref{eq:kant_ot_d_2} admits a minimizer $\phi_0$, which can be shown to be a convex function. In fact, we obtain a pair of \emph{Brenier potentials}, written
\begin{align*}
    (\phi_0,\psi_0) \defeq (\phi_0, \phi_0^*) = (\psi_0^*, \psi_0)\,.
\end{align*}
Furthermore, the pair $(\phi_0,\psi_0)$ is unique up to additive constants $(\phi_0+c,\psi_0-c)$.

Recalling \eqref{eq:w22_intro}, the following theorem unifies these three formulations of optimal transport under the squared-Euclidean cost.
\begin{theorem}[\citealp{Bre91}] \label{thm:brenier}
Let $\rho, \nu \in \cP_{2}$. Assume that $\rho$ has  a density with respect to the Lebesgue measure. Then,
\begin{enumerate}
\item the solution to \eqref{eq:w22_intro} exists and is of the form $T_0 = \nabla \phi_0 = \nabla \psi^*_0$, with $\phi_0$ convex being a solution to \eqref{eq:kant_ot_d_2}.
\item $\pi_0$ is also uniquely defined as  $$\dd \pi_0(x,y) = \dd \rho(x) \delta_{\{\nabla \phi_0(x)\}}(y)\,.$$
\end{enumerate}
\end{theorem}
If $\nu$ also has a density, then the \emph{backward} Brenier map, from $\nu$ to $\rho$, is given by
\begin{align*}
    (T_0)^{-1} = \nabla\psi_0 = \nabla\phi_0^*\,.
\end{align*}
To be concise, we simply write $\phi_0^\nu$, $\psi_0^{\nu}$ and $T_0^\nu$ to refer to the optimal transport quantities associated to the problem $W_2^2(\rho,\nu)$ (and similarly for $W_2^2(\rho,\mu)$).

\subsection{Entropic optimal transport}\label{sec: background_eot}
For two probability measures $\rho,\nu \in \cP_2$, the entropic optimal transport objective
\citep{cuturi2013sinkhorn} is defined as 
\begin{align}\label{eq:kant_eot_p}
     \tfrac12W_{2,\eps}^2(\rho,\nu) \defeq \min_{\pi \in \Gamma(\rho,\nu)} \iint \tfrac12\|x - z\|^2 \dd \pi(x,z) + \eps \kl{\pi}{\rho\otimes\nu}\,,
\end{align}
for some $\eps > 0$, and $\kl{\pi}{\rho\otimes\nu}$ is the Kullback--Leibler divergence, defined as
\begin{align*}
    \kl{\pi}{\rho\otimes\nu} \defeq \int \log\left(\frac{\dd \pi}{\dd \rho \otimes \dd \nu}\right)\dd \pi
\end{align*}
when $\pi$ is absolutely continuous with respect to $\rho \otimes \nu$, and $+\infty$ otherwise. Note that due to the regularization term, the problem is strictly convex with a unique minimizer $\pi_\eps^\nu$, the \emph{optimal entropic (transport) coupling}.\footnote{Since $\rho$ is fixed throughout, we use the superscript $\nu$ (respectively, $\mu$) to indicate objects that correspond to the entropic optimal transport problem between $\rho$ and $\nu$ (respectively, $\rho$ and $\mu$).}

The entropic optimal transport problem also admits a dual formulation \citep[see, e.g.,][]{genevay2019entropy}:
\begin{align}\label{eq:kant_eot_d_2}
    \tfrac12W_{2,\eps}^2(\rho,\nu) \defeq \tfrac12M_2(\rho+\nu) - \min_{\phi \in L^1(\rho)} \int \phi \dd \rho + \int \Phi_\eps^\rho[\phi] \dd \nu\,,
\end{align}
where $\Phi_\eps^\nu$ is the following operator
\begin{align*}
   \forall z\in\R^d,\ \Phi_\eps^\rho[\phi](z) \defeq \eps \log \int e^{(\langle x,z \rangle - \phi(x))/\eps}\dd \rho(x)\,,
\end{align*}
which should be thought of as the entropic analogue to the convex conjugate operator. Indeed, notice that as $\eps \to 0$, $\Phi_\eps^\rho[\phi](z) $ converges to the $\rho$-essential supremum of the function $x\mapsto \langle x,z \rangle - \phi(x)$.
We write the minimizer to \eqref{eq:kant_eot_d_2} as $\phi_\eps^\nu$, from which we obtain the minimizing pair of \emph{entropic Brenier potentials}
\begin{align*}
    (\phi_\eps^\nu,\psi_\eps^\nu) \defeq (\phi_\eps^\nu,\Phi_\eps^\rho[\phi_\eps^\nu]) = (\Phi_\eps^\nu[\psi_\eps^\nu],\psi_\eps^\nu)\,,
\end{align*}
where $\Phi_\eps^\nu$ is defined analogously to $\Phi_\eps^\rho$. Again, this pair is unique up to constant shifts. 

Moreover, by the dual optimality conditions, we can define versions of the entropic Brenier potentials taking values in the extended reals, for all $x \in \Rd$ and $z \in \Rd$, respectively. Thus, we freely write
\begin{equation}\label{eq:dual_entropic}
        \begin{split}
        &\phi_\eps^\nu(x) \defeq  \eps \log \int e^{(\langle x,z \rangle - \psi_\eps^\nu(z))/\eps}\dd \nu(z) \quad (x \in \R^d) \\
        &\psi_\eps^\nu(z) \defeq \eps \log \int e^{(\langle x,z \rangle - \phi_\eps^\nu(x))/\eps}\dd \rho(x) \quad (z \in \R^d)\,.
    \end{split}
\end{equation}
See \cite{mena2019statistical,nutz2021entropic} for more details. Note that $\phi_\eps^\nu : \R^d \to \R\cup\{+\infty\}$ (resp. $\psi_\ep^\nu$) is a convex function which is analytic on the interior of its domain $\dom(\phi_\eps^\nu)$ (resp. $\dom(\psi_\eps^\nu)$),

An important feature of the entropic optimal transport problem is that the optimal solutions to~\eqref{eq:kant_eot_p} and~\eqref{eq:kant_eot_d_2} satisfy the following primal-dual relationship \citep{Csi75}:
\begin{align*}
    \dd \pi_\eps^{\nu}(x,z) \defeq \gamma_\eps^{\nu}(x,z) \dd\rho(x)\dd\nu(z) \defeq e^{(\langle x,z\rangle - \psi_\eps^\nu(z) - \phi_\eps^\nu(x))/\eps}\dd\rho(x)\dd\nu(z)\,.
\end{align*}
Concretely, $\gamma_\eps^{\nu}$, the density of $\pi_\eps^{\nu}$ with respect to $\rho \otimes \nu$, can be written explicitly in terms of the entropic Brenier potentials $(\phi_\eps^\nu,\psi_\eps^\nu)$.

Let $(X, Z)$ be a pair of random variables with distribution $\pi_\eps^\nu$.
For a given $x \in \dom(\phi_\eps^\nu)$, we abuse notation and define the conditional probability of $Z|X = x$ as 
\begin{align*}
    \dd \pi_\eps^\nu(z|x) = e^{(\langle x,z \rangle - \phi_\eps^\nu(x) - \psi_\eps^\nu(z))/\eps}\dd\nu(z)\,.
\end{align*}
 Similarly we denote $\pi_\eps^\nu(\cdot|z) \defeq \pi_\eps^\nu(\cdot|Z = z)$ whenever $z\in \dom(\psi_\eps^\nu)$.
Likewise, if $(X, Y)$ are distributed according to the optimal entropic coupling $\pi_\eps^\mu$ between $\rho$ and $\mu$, we will write $\pi_\eps^\mu(\cdot|x)$ and $\pi_\eps^\mu(\cdot|y)$ for the conditional distributions of $Y |X = x$ and $X|Y =y$, respectively.
We will adopt the convention throughout that $X$, $Y$, and $Z$ always refer to random variables with marginal distributions $\rho$, $\mu$, and $\nu$, respectively.

Following e.g., \cite{pooladian2021entropic,rigollet2022sample}, we define, respectively, the \emph{forward} and \emph{backward entropic Brenier maps} from $\rho$ to $\nu$ to be barycentric projections of $\pi_\eps^{\nu}$ \citep[Definition 5.4.2]{Ambrosio2008}: 
for $x\in \dom(\phi_\eps^\nu)$ and $z\in \dom(\psi_\eps^\nu)$, we define
\begin{equation*}
	 T_\eps^\nu(x) \defeq \int z \dd \pi_\eps^\nu(z|x)\,, \quad S_\eps^\nu(z) \defeq \int x \dd \pi_\eps^\nu(x|z)\,
\end{equation*}
whenever the integrals are well-defined. Unlike the unregularized case, $(S_\eps^\nu)^{-1} \neq T_\eps^\nu$. Note that by Jensen's inequality, $T_\eps^\nu \in L^2(\rho)$ with $\|T_\eps^\nu\|_{L^2(\rho)}^2\leq M_2(\nu)$ (resp. $S_\eps^\nu \in L^2(\nu)$ with $\|S_\eps^\nu\|_{L^2(\nu)}^2\leq M_2(\rho)$). Also note 
that by the dominated convergence theorem, the gradient of $\phi_\eps^\nu$ (resp. $\psi_\eps^\nu$) from \eqref{eq:dual_entropic} has a natural interpretation as the forward (resp. backward) entropic Brenier map: whenever $x$ is in the interior of $\dom(\phi_\eps^\nu)$ and $z$ is in the interior of $\dom(\psi_\eps^\nu)$,
\begin{equation}\label{eq:gradients}
	 \nabla \phi_\eps^\nu(x) = T_\eps^\nu(x)\,, \quad \nabla\psi_\eps^\nu(z) = S_\eps^\nu(z)\,.
\end{equation}
 Under the same condition, a similar expression holds for the Hessians of the entropic Brenier potentials \citep[see, e.g.,][Lemma 1]{chewi2023entropic}:
\begin{align}\label{eq:hessians}
    \nabla^2\phi_\eps^\nu(x) = \eps^{-1}\text{Cov}_{\pi_\eps^\nu}(Z|X=x)\,, \quad \nabla^2\psi_\eps^\nu(z) = \eps^{-1}\text{Cov}_{\pi_\eps^\nu}(X|Z=z)\,.
\end{align}

Throughout, we will write $(\phi_\eps^\nu,\psi_\eps^\nu)$ for the entropic Brenier potentials associated to $\tfrac12\OTep(\rho,\nu)$,  $T_\eps^\nu$ for the forward entropic Brenier map, and $S_\eps^\nu$ for the backward entropic Brenier map (and similarly for $\tfrac12\OTep(\rho,\mu)$). 

\subsubsection*{Related work in entropic optimal transport}
\paragraph{Fixed regularization.} The initial motivation for studying \eqref{eq:kant_eot_p} in the machine learning literature was its significant computational benefits compared to \eqref{eq:kant_ot_p} \citep{cuturi2013sinkhorn,AltWeeRig17}. As a result, the study of entropic objects for a fixed $\eps>0$ regularization parameter has been of great interest in a number of fields. For example, \cite{del2022improved,goldfeld2022limit,gonzalez2022weak} studied statistical limit theorems for entropic optimal transport. \cite{greco2023non,
conforti2023quantitative, nutz2023stability} study the convergence of Sinkhorn's algorithm to the optimal Brenier potentials at the population level. The works by \cite{masud2021multivariate,pooladian2022debiaser,rigollet2022sample,stromme2023minimum,klein2023learning,werenski2023estimation} studied additional computational or statistical properties of entropic Brenier maps. As previously mentioned, \cite{carlier2022lipschitz} initiated the study of the stability properties of entropic Brenier maps under variations of the target measure, though their techniques differ significantly from ours.

\paragraph{Vanishing regularization.} Theoretical properties of entropic optimal transport for vanishing regularization parameter are widely studied in both statistical and theoretical works. For example, convergence of the regularized to unregularized couplings was studied by \citet{leonard2012schrodinger,carlier2017convergence, bernton2021entropic, ghosal2021stability}, and convergence of the transport costs by \cite{chizat2020faster,conforti2021formula,eckstein2023convergence,pal2024difference}. \citet{nutz2021entropic} established convergence of the entropic to non-entropic Brenier potentials under minimal assumptions; this convergence was improved in the case of semi-discrete optimal transport by \citet{altschuler2021asymptotics} and \citet{delalande2021nearly}. \citet{chewi2023entropic} established a short proof of Caffarelli's contraction theorem \citep{caffarelli2000monotonicity} via covariance inequalities and entropic optimal transport, which was subsequently generalized by \cite{conforti2022weak}. Statistical convergence of entropic Brenier maps to unregularized Brenier maps was established by \cite{pooladian2021entropic,pooladian2023minimax}, the latter paper focusing on the semi-discrete setting.

\subsection{A transport inequality for conditional entropic couplings}
At the core of our approach is the use of a specific transport inequality which has been developed for other purposes in the study of sampling and functional inequalities \citep{anari2021spectral,anari2021entropic,Chen2022,bauerschmidt2023stochastic}. We refer to \citet[Section 3.7]{bauerschmidt2023stochastic} for more details, and briefly overview the necessary inequalities and notation here.

Let $q \in \cP_2$ be a probability measure with finite moment generating function whose covariance is denoted by $\text{Cov}(q)$. For $h \in \R^d$, we define the tilt  $\mathcal{T}_hq$ of $q$ as the probability measure satisfying
\begin{align*}
  \forall z \in \R^d,\quad   \frac{\dd \mathcal{T}_hq}{\dd q}(z) \defeq \frac{\exp(\langle h,z\rangle)}{\E_{Z\sim q}[\exp(\langle h , Z \rangle)]}\,.
\end{align*}
We say that $q$ is \emph{tilt-stable}\footnote{This name is not standard, but we introduce it here because the standard name for this concept (entropic stability) is likely to cause confusion in the context of our main results.} if for all $h\in \R^d$, $\text{Cov}(\mathcal{T}_hq) \preceq C_{\mathsf{T}} I$ for some $C_{\mathsf{T}} > 0$. If $q$ is tilt-stable with constant $C_{\mathsf{T}}$, then for all probability measures $p \in \cP_2$,
\begin{align*}
    \|\E_{p}[X] - \E_{q}[X]\|^2\leq 2C_{\mathsf{T}}\kl{p}{q}\,,
\end{align*}
see \cite[Lemma 3.21]{bauerschmidt2023stochastic}.

Our main observation is that conditional entropic couplings are tilt-stable, with a constant that can be written in terms of the entropic Brenier potentials.
For an entropic potential $\phi_\eps^\nu$ whose domain is all of $\R^d$, we write
\begin{equation}\label{eq:hmax}
    H_{\max}(\phi_\eps^\nu) \defeq \sup_{u\in\R^d} \|\text{Cov}_{\pi_\eps^\nu}(Z|X=u)\|_{\text{op}} = \eps \sup_{u\in\R^d} \|\nabla^2 \phi_\eps^\nu(u)\|_{\text{op}} \,,
\end{equation}
and define $H_{\max}(\psi_\eps^\nu)$ analogously.
The second equality in~\eqref{eq:hmax} is justified by the fact that the identity~\eqref{eq:hessians} holds everywhere when $\dom(\phi_\eps^\nu) = \R^d$.
If either potential is not finite on all of $\R^d$, we adopt the convention that $H_{\max} = +\infty$.
\begin{lemma}\label{lem:tiltcondcouplings}
    Let $x\in \R^d$ and let $\pi_\eps^\nu(\cdot|x)$ be a conditional entropic coupling between two probability measures $\rho, \nu \in \cP_2$. Assume that $\dom(\phi_\eps^\nu)=\R^d$. Then for any $h \in \R^d$, 
    \begin{align*}
        \mathcal{T}_h\pi_\eps^\nu(\cdot|x) = \pi_\eps^\nu(\cdot|x+\eps h)\,.
    \end{align*}
\end{lemma}
\begin{proof}
We temporarily omit the superscript in $\nu$ for ease of reading. Note that
    \begin{align*}
        \E_{Z\sim \pi_\eps(\cdot|x)}[e^{\langle h, Z\rangle }] = e^{-\phi_\eps(x)/\eps}\int e^{( \langle x+\eps h,  z \rangle - \psi_\eps(z))/\eps}\dd\nu(z) = e^{(\phi_\eps(x+\eps h) - \phi_\eps(x))/\eps}\,.
    \end{align*}
    From this we can conclude, since for all $z\in \R^d$
    \begin{align*}
        \frac{\dd\cT_h\pi_\eps(\cdot|x)}{\dd\nu}(z) &= e^{( \langle x + \eps h , z \rangle  - \phi_\eps(x) - \psi_\eps(z))/\eps} e^{(\phi_\eps(x) - \phi_\eps(x+\eps h))/\eps} \\
        &= e^{( \langle x + \eps h , z \rangle  - \phi_\eps(x+\eps h) - \psi_\eps(z))/\eps} \\
        &= \frac{\dd\pi_\eps(\cdot|x+\eps h)}{\dd \nu}(z)\,. \qedhere
    \end{align*}
\end{proof}
\begin{corollary}\label{cor:etransport}
The conditional entropic coupling $\pi_\eps^\nu(\cdot|x)$ (resp. $\pi_\eps^\nu(\cdot|z)$) is tilt-stable with constant $H_{\max}(\phi_\eps^\nu)$ (resp. $H_{\max}(\psi_\eps^\nu)$).
\end{corollary}
\begin{proof}
If the domain of $\phi_\eps^\nu$ is not $\R^d$, then $H_{\max}(\phi_\eps^\nu)=+\infty$, and the proposition is vacuous. Otherwise, fix $x \in \R^d$. By definition of tilt stability, it suffices to compute an upper bound on the covariance of $\mathcal{T}_h\pi^\nu_\eps(\cdot|x)$ which holds uniformly over all tilts $h \in \R^d$. This follows by direct computation, as \cref{lem:tiltcondcouplings} and \eqref{eq:hmax} imply that
    \begin{align*}
        \text{Cov}(\cT_h \pi_\eps^\nu(\cdot|x)) = \text{Cov}(\pi_\eps^\nu(\cdot|x+\eps h)) 
        \preceq H_{\max}(\phi_\eps^\nu)I\,, 
    \end{align*}
    where the last inequality holds by taking the supremum over both $x$ and $h$ arguments. Note that the argument is symmetric for the other conditional entropic coupling. 
\end{proof}

\section{Main results}\label{sec: improved_stability}
We now present our general stability result for entropic Brenier maps. 

\begin{theorem}[Stability of entropic Brenier maps]\label{thm:stab_gen}
Suppose $\rho,\mu,\nu$ have finite second moment. Then
    \begin{align*}
            \| T_\eps^\mu - T_\eps^\nu\|_{L^2(\rho)} \leq \Bigl(1 + \frac{2(H_{\max}(\phi_\eps^\nu) H_{\max}(\psi_\eps^\nu))^{1/2}}\eps \Bigr)W_2(\mu,\nu)\,.
    \end{align*}
\end{theorem}

\begin{remark}
Note that if the potentials $(\phi_\eps^\nu,\psi_\eps^\nu)$ are not finite everywhere, the quantities $H_{\max}(\phi^\nu_\eps)$ and $H_{\max}(\psi^\nu_\eps)$ are infinite by convention, and the inequality becomes vacuous.  The potentials are finite everywhere whenever $\rho$ and $\nu$ have moment-generating functions that are finite everywhere (including the important case of bounded supports), but also when $\rho$ and $\nu$ have support equal to $\R^d$, without additional tail assumptions. 
\end{remark}

From here, we can prove the results highlighted in the introduction as special cases.
\vspace{-3mm}
\begin{corollary}[Entropic stability for bounded measures]\label{cor:stab_1} Suppose $\rho$ and $\nu$ are supported in $B(0;R)$, and $\mu$ has finite second moment. Then 
\begin{align*}
\| T_\eps^\mu - T_\eps^\nu\|_{L^2(\rho)} \leq \Bigl(1 + \frac{2R^2}{\eps}\Bigr)W_2(\mu,\nu)\,.
\end{align*}
\end{corollary}
\begin{proof}
Since $\rho$ and $\nu$ are supported in $B(0;R)$,   \eqref{eq:hessians} implies that both $H_{\max}(\phi_\eps^\nu)$ and $H_{\max}(\psi_\eps^\nu)$ are smaller than $R^2$, which completes the proof.  
\end{proof}
Note that \cref{cor:stab_1} is \emph{entirely general} in its requirements, and does not rely on smoothness of maps, nor do any of the measures require densities.

\begin{remark}[Tightness of \cref{cor:stab_1}]\label{rmk:tight}
We now demonstrate that \cref{cor:stab_1} is tight for general bounded probability measures. Fix $R > 0$, and let $p_\theta \defeq \tfrac12\delta_{Re_\theta} + \tfrac12\delta_{-Re_\theta}$ with $e_\theta = (\cos(\theta),\sin(\theta))$, for $\theta \in [0,\tfrac{\pi}{2}]$. Let $\rho \defeq p_{\pi/2}$ and $\eps > 0$, and let $\pi_\eps^\theta$ denote the entropic optimal coupling between $\rho$ and $p_\theta$ for $\theta \in [0,\tfrac{\pi}{2})$. One can deduce that the optimal entropic coupling is symmetric for any such $\theta$: 
\begin{align*}
    \pi_\eps^\theta(x,y) = \pi_\eps^\theta(-x,-y)\,,
\end{align*}
Following the calculations in \citet[Section 3]{altschuler2021asymptotics}, one can choose $\psi_\eps^\theta(e_\theta)=\psi_\eps^\theta(-e_\theta)=0$, so that by \eqref{eq:dual_entropic}, for all $x\in \R^d$,
\begin{align*}
    \phi_\eps^\theta(x) = \eps \log \bigl(\tfrac12e^{R\langle x, e_\theta \rangle/\eps} + \tfrac12e^{R\langle x, -e_\theta \rangle/\eps}\bigr)\,.
\end{align*}
 Then we compute
\begin{align*}
    T_\eps^\theta(x) = Re_\theta\bigl(\pi_\eps^\theta(x,e_\theta) - \pi_\eps^\theta(x,-e_\theta)\bigr) = Re_\theta \tanh({R\langle x, e_\theta\rangle}/\eps)\,.
\end{align*}
Let $\mu \defeq p_0$ and $\nu\defeq p_\theta$. Following the above calculations, it is clear that
\begin{align*}
    \|T_\eps^\mu - T_\eps^\nu\|_{L^2(\rho)} 
    = R \|e_\theta\| \sqrt{\tanh^2(R\sin(\theta)/\eps)} = \frac{R^2\theta}{\eps} + O(\theta^4)\,.
\end{align*}
It is also easy to verify that $W_2(p_0,p_\theta) \asymp \theta$, since the optimal transport map from $p_0$ to $p_\theta$ is the standard $2 \times 2$ rotation matrix acting on the dirac masses. This example shows that for $\theta$ small the dependence $R^2 \eps^{-1}$ in \cref{cor:stab_1} is tight.
\end{remark}

The following example provides the entropic analogue of Theorem 6 from \citet{manole2021plugin}; their result is formally recovered in the $\eps \to 0$ limit.

\begin{corollary}[Improved stability under smoothness]\label{cor:stab_2}
Suppose $T_\eps^\nu$ is uniformly $\Lambda$-Lipschitz. If $\rho$ is supported in $B(0;R)$, then
\begin{align*}
\|T_\eps^\mu - T_\eps^\nu\|_{L^2(\rho)} \leq \bigl( 1 + 2\sqrt{\Lambda R^2/\eps} \bigr) W_2(\mu,\nu) \,. 
\end{align*}
If instead $S_\eps^\nu$ is uniformly $1/\lambda$-Lipschitz, then
\begin{align*}
\|T_\eps^\mu - T_\eps^\nu\|_{L^2(\rho)} \leq (1 + 2\sqrt{\Lambda/\lambda}) W_2(\mu,\nu) \,.
\end{align*}
\end{corollary}
\begin{proof}
The first claim follows from the bounds $H_{\max}(\phi_\eps^\nu) \leq \Lambda\eps $, which follows from \eqref{eq:hmax}, and $H_{\max}(\psi_\eps^\nu)\leq R^2$. For the second, we instead use that $H_{\max}(\psi_\eps^\nu) \leq \eps/\lambda $.
\end{proof}

\subsection{Proof of Theorem~\ref{thm:stab_gen}}\label{sec:entstab_proof_sketch}
Our proof relies on three propositions.
To continue, we require the following objects. Let $\tau \in \Gamma(\mu,\nu)$ be a fixed (though not necessarily unique) optimal transport coupling between $\mu$ and $\nu$. For $z\in \R^d$, let $\tau(\cdot|z)$ be associated (regular) conditional measure \citep[see][Chapter 10]{bogachev2007measure}, so that for all measurable $f:\R^d\times \R^d \to [0,+\infty)$
\[ \iint f(y,z)\dd\tau(y,z) = \int \p{\int f(y,z)\dd \tau(y|z)}\dd\nu(z)\,.\]
For $x\in \R^d$, let $Q(\cdot|x)$ be the probability measure with
\begin{align}\label{eq:def_Qx}
   \forall z\in \R^d,\ \frac{ \dd Q(\cdot|x)}{\dd\nu}(z) \defeq \int \gamma_\eps^\mu(x,y)\dd\tau(y|z)\,,
\end{align}
where $\gamma_\eps^\mu(x,y)$ is the density of $\pi_\eps^\mu$ w.r.t.~$\rho\otimes\mu$. Note that this indeed defines a density as we have the relation
\[
\int  \p{\int \gamma_\eps^\mu(x,y)\dd \tau(y|z)}\dd\nu(z)= \iint \gamma_\eps^\mu(x,y)\dd \tau(y,z)=\int \gamma_\eps^\mu(x,y)\dd\mu(y)=1\,.
\]

We also define the conditional Kullback--Leibler divergence:
\begin{align}\label{eq:def_I}
    I \defeq \int \kl{Q(\cdot|x)}{\pi_\eps^\nu(\cdot|x)}\dd\rho(x)\,,
\end{align}
We are now in a position to proceed with the proof. First, we decompose the difference of forward entropic Brenier maps into a $W_2(\mu,\nu)$ term, plus a term depending on $I$.
\begin{prop}\label{prop:step1}
Suppose $\rho,\mu,\nu$ have finite second moment. Then
\begin{align*}
    \|T_\eps^\mu - T_\eps^\nu\|_{L^2(\rho)} \leq W_2(\mu,\nu) + (2 \Hphi I)^{1/2}\,.
\end{align*}
\end{prop}

\begin{proof}[Proof of \cref{prop:step1}]
We assume that $\phi_\eps^\nu$ is finite everywhere, for otherwise there is nothing to prove. 
We have that
\begin{align*}
&T_\eps^\mu(x)- T_\eps^\nu(x) = \int y\dd \pi_\eps^\mu(y|x) - \int z \dd \pi_\eps^\nu(z|x) \\
&= \iint y\gamma_\eps^\mu(x,y) \dd \tau(y,z) - \int z \dd \pi_\eps^\nu(z|x) \\
&= \iint (y-z) \gamma_\eps^\mu(x,y) \dd \tau(y,z) + \iint z (\gamma_\eps^\mu(x,y)\dd \tau(y,z) - \dd \pi_\eps^\nu(z|x)) \\
&= \iint (y-z) \gamma_\eps^\mu(x,y) \dd \tau(y,z) + \int z \dd(Q(z|x) - \pi_\eps^\nu(z|x))\,.
\end{align*}

Taking the $L^2(\rho)$-norm of both sides and applying Minkowski's and Jensen's inequalities yields
\begin{multline*}
    \| T_\eps^\mu- T_\eps^\nu \|_{L^2(\rho)} \leq \left(\iint \|y-z\|^2 \gamma_\eps^\mu(x,y)\dd\tau(y,z) \dd \rho(x)\right)^{1/2} \\
      + \Bigl\| \int z \dd(Q(z|\cdot) - \pi_\eps^\nu(z|\cdot)) \Bigr\|_{L^2(\rho)}\,.
\end{multline*}
Since $\tau$ is an optimal coupling between $\mu$ and $\nu$ and $\int \gamma_\eps^\mu(x,y) \dd \rho(x)=1$, the first term is $W_2(\mu, \nu)$.
For the second term, \cref{cor:etransport} implies for all $x \in \RR^d$
\begin{equation*}
	\Bigl\| \int z \dd(Q(z|x) - \pi_\eps^\nu(z|x)) \Bigr\| \leq \sqrt{2\Hphi\kl{Q(\cdot|x)}{\pi_\eps^\nu(\cdot|x)}}\,.
\end{equation*}
Therefore
\begin{align*}
	\Bigl\| \int z \dd(Q(z|\cdot) - \pi_\eps^\nu(z|\cdot)) \Bigr\|_{L^2(\rho)} & \leq \sqrt{2 \Hphi\int \kl{Q(\cdot|x)}{\pi_\eps^\nu(\cdot|x)} \dd \rho(x)} \\
	& = (2 \Hphi I)^{1/2}\,,
\end{align*}
which completes the proof.
\end{proof}

In \eqref{eq:def_I}, we defined the conditional relative entropy between $Q(\cdot|x)$ and the conditional entropic coupling $\pi_\eps^\nu(\cdot|x)$. We now turn to directly bounding the quantity $I$.
Note that for all $z\in \R^d$
\begin{align*}
\frac{\dd Q(\cdot|x)}{\dd\pi_\eps^\nu(\cdot|x)}(z) =\frac{\int \gamma_\eps^\mu(x,y)\dd\tau(y|z)}{\gamma_\eps^\nu(x,z)}\,.
\end{align*}
An application of Jensen's inequality then yields that
\begin{align}\label{eq:i_to_i_bar}
   I \leq \bar{I} \defeq \iiint  \log\Bigl(\frac{\gamma_\eps^\mu(x,y)}{\gamma_\eps^\nu(x,z)}\Bigr) \gamma_\eps^\mu(x,y)\dd\tau(y,z)\dd\rho(x)\,.
\end{align}
Expanding the densities $\gamma_\eps^\mu(x,y)$ and $\gamma_\eps^\nu(x,z)$ and performing the integration, we obtain
\begin{align*}
\eps \bar{I} &= \iiint \langle x ,y-z \rangle \gamma_\eps^\mu(x,y)\dd \tau(y,z)\dd\rho(x)  + \int \phi_\eps^\nu \dd\rho + \int \psi_\eps^\nu \dd\nu - \int \phi_\eps^\mu\dd\rho - \int \psi_\eps^\mu\dd\mu \\
&= \iint \langle S_\eps^\mu(y), y-z\rangle \dd \tau(y,z)+ \int \phi_\eps^\nu \dd\rho + \int \psi_\eps^\nu \dd\nu - \int \phi_\eps^\mu\dd\rho - \int \psi_\eps^\mu\dd\mu
\end{align*}
where we use the equality $S_\eps^\mu(y)=\int x\gamma_\eps^\mu(x,y)\dd \rho(x)$ in the last line. 
If we define $\tilde{I}$ as a symmetric analogue to $\bar{I}$, namely,
\begin{align*}
    \tilde{I} \defeq \iiint  \log\Bigl(\frac{\gamma_\eps^\nu(x,z)}{\gamma_\eps^\mu(x,y)}\Bigr) \gamma_\eps^\nu(x,z)\dd\tau(y,z)\dd\rho(x)\,,
\end{align*}
then, since $0 \leq \tilde{I}$, a symmetric calculation immediately yields the following: 
\begin{prop}\label{prop:step2} Suppose $\rho,\mu,\nu$ have finite second moment. Then
\begin{align}\label{eq:barIbound}
    \eps \bar{I} \leq \eps(\bar{I}+\tilde{I}) = \iint \langle S_\eps^\mu(y) - S_\eps^\nu(z), y-z\rangle \dd \tau(y,z)\,.
\end{align}
\end{prop}
All in all, our bound currently reads
\begin{align*}
    \|T_\eps^\mu -T_\eps^\nu\|_{L^2(\rho)} &\leq W_2(\mu,\nu) \\
    &\qquad + (2 \eps^{-1} \Hphi)^{1/2} \Bigl(\iint \langle S_\eps^\mu(y) - S_\eps^\nu(z), y-z\rangle \dd \tau(y,z) \Bigr)^{1/2}\,.
\end{align*}
The second term depends on the two backward entropic Brenier maps, $S_\eps^\mu$ and $S_\eps^\nu$.
The next proposition shows that tilt stability can again be used to bound the difference between these backward maps by $\bar I$.
\begin{prop}\label{prop:step3}
Suppose $\rho,\mu,\nu$ have finite second moment.
    Then
    \begin{align*}
        \iint \|S_\eps^\mu(y) - S_\eps^\nu(z) \|^2 \dd \tau(y,z)  \leq 2 \Hpsi \bar{I}\,.
    \end{align*}
\end{prop}
\begin{proof}
We assume that $\psi_\eps^\nu$ is finite everywhere, for otherwise there is nothing to prove. 
Recall that $S_\eps^\nu(z) = \int x \dd \pi_\eps^\nu(x|z)$ and similarly for $S_\eps^\mu(y)$. By \cref{cor:etransport}, we have the following bound
\begin{align*}
    \| S_\eps^\nu(z) - S_\eps^\mu(y) \|^2 &= \Bigl\| \int x \dd(\pi_\eps^\nu(x|z) - \pi_\eps^\mu(x|y))\Bigr\|^2 \\
    &\leq 2 \Hpsi\kl{\pi_\eps^\mu(\cdot|y)}{\pi_\eps^\nu(\cdot|z)} \\
    &= 2 \Hpsi \int \log\Bigl(\frac{\gamma_\eps^\mu(x,y)}{\gamma_\eps^\nu(x,z)}\Bigr)\gamma_\eps^\mu(x,y)\dd\rho(x)\,.
\end{align*} 
Integrating with respect to $\tau$ concludes the proof.
\end{proof}
Finally, with these results in hand, we can prove our main result.

\begin{proof}[Proof of \cref{thm:stab_gen}]
From \eqref{eq:barIbound}, we apply Cauchy-Schwarz, resulting in
\begin{align*}
    \eps \bar{I} &\leq \iint \langle S_\eps^\mu(y) - S_\eps^\nu(z), y-z\rangle \dd \tau(y,z) \leq W_2(\mu,\nu)(2 H_{\max}(\psi_\eps^\nu) \bar{I})^{1/2}\,.
\end{align*}
This ultimately implies 
\begin{align*}
    I^{1/2} \leq \bar{I}^{1/2} \leq \eps^{-1} W_2(\mu,\nu)(2H_{\max}(\psi_\eps^\nu))^{1/2}\,,
\end{align*}
where we recall the first inequality from \eqref{eq:i_to_i_bar}. Together with \cref{prop:step1}, the proof is complete.
\end{proof}
\section{Application: Improved quantitative stability of semi-discrete optimal transport maps}\label{sec:quantstab_semidiscrete}
As an application of our new stability results for entropic Brenier maps, we turn to proving quantitative stability results for \emph{unregularized} optimal transport maps. As highlighted in the introduction, our proof technique for proving quantitative stability for optimal transport maps is based on the following decomposition:
\begin{align}\label{eq:bias_variance_decomp}
\begin{split}
    \| T_0^{\mu}-T_0^{\nu}\|_{L^2(\rho)} &\leq \| T_0^{\mu} - T_\eps^{\mu}\|_{L^2(\rho)} + \| T_0^{\nu} - T_\eps^{\nu}\|_{L^2(\rho)} \\
    &\qquad + \| T_\eps^{\mu} - T_\eps^{\nu}\|_{L^2(\rho)}\,.
\end{split}
\end{align}
Recall that \cref{cor:stab_1} takes care of the last term under virtually no assumptions other than boundedness of the measures. It remains to control the first two terms in the above decomposition, also known as \emph{bias} terms.

To the best of our knowledge, bounds on the bias of entropic Brenier maps are known only under strong assumptions.
For example, \citet[Corollary 1]{pooladian2021entropic} show that
\begin{align*}
    \| T_0^{\mu} - T_\eps^{\mu}\|_{L^2(\rho)}^2 \lesssim \eps^2 I_0(\rho,\mu)\,,
\end{align*}
where $I_0(\rho,\mu)$ is the integrated Fisher information along the Wasserstein geodesic between $\rho$ and $\mu$, where $\rho$ and $\mu$ have upper and lower bounded densities over compact domains. Such assumptions, while essential for estimating optimal transport maps on the basis of samples, are too restrictive for our purposes.\footnote{Indeed, under these assumptions, it is well-known via Caffarelli regularity theory \citep{caffarelli1992boundary,
caffarelli1996boundary} that the corresponding optimal transport map is Lipschitz, so the results of \cite{gigli2011holder} already imply a stability bound.}

With this in mind, our goal is to establish quantitative control on the bias of entropic Brenier maps under less restrictive regularity conditions. Specifically, we turn to the \emph{semi-discrete} setting, where $\rho$ has a density, and $\mu$ and $\nu$ are both discrete measures.  As we will shortly see, this setting allows for \eqref{eq:bias_variance_decomp} to be used to obtain meaningful bounds on the stability of two semi-discrete optimal transport maps when the discrete measures themselves have favorable properties. 

We briefly recall some fundamental notions from semi-discrete optimal transport: let $\mu = \sum_{j=1}^J \mu_j \delta_{y_j}$ be a discrete probability measure with atoms located at the points $\{y_j\}_{j=1}^J$ with corresponding weights $\mu_j > 0$. In this setting, the Brenier potential is given explicitly by 
\begin{align*}
    \phi_0^\mu(x) = \max_{j \in \{1,\ldots,J\}} \langle x, y_j\rangle - (\psi_0^\mu)_j\,,
\end{align*}
where $\psi_0^\mu \in \R^J$ is the dual potential. Note that $\phi_0^\mu$ is $\rho$-almost everywhere differentiable, and so the Brenier map $T_0^\mu = \nabla \phi_0^\mu$ is well-defined. The inverse transport map is now set-valued, where for a given target atom $y_j$, we define the \emph{Laguerre cell} $L_j \defeq (T_0^\mu)^{-1}(y_j)$. These cells partition the support of $\rho$. Consequently, for $x \in L_j$, the optimal transport mapping is $x \mapsto T_0^\mu(x) = y_j$.

With these notions in hand, we are ready to present the following result on the convergence of entropic Brenier maps to their unregularized counterpart; its proof is located in \cref{sec: biases}. This is a slightly different version of \citet[Theorem 3.4]{pooladian2023minimax}, based off the results of  \citet{delalande2021nearly}.
\begin{prop}[Quantitative bias in the semi-discrete setting]\label{prop:bias_main}
    Let $\rho$ be a compactly supported probability distribution with a density over $\R^d$ and $\mu$ be a discrete measure, written $\mu = \sum_{j=1}^J\mu_j \delta_{y_j}$. Then for all $\eps > 0$, 
    \begin{align}\label{eq:biasbound_prelim}
        \| T_0^\mu - T_\eps^\mu\|_{L^2(\rho)}^2 \leq e^{2\|\psi_0^\mu - \psi_\eps^\mu\|_\infty/\eps}\eps \sum_{i,j} \frac{\|y_i - y_j\|}{2} \int_{0}^\infty \!\!\!h_{ij}(u\eps)\Bigl( 1 + \tfrac{\mu_i}{\mu_j}e^{u/2} \Bigr)^{-1} \dd u\,,
    \end{align}
    where $0 \leq h_{ij}(\cdot)$ measures the amount of overlap between $L_i$ and $L_j$ weighted against the source measure $\rho$ (see \eqref{eq:defhij} in the appendix for precise details). 
    
    In addition, suppose that
    \begin{description}
        \item \textbf{(S1)} the density of $\rho$ has convex support  in $B(0;R)$, is $\alpha$-H{\"o}lder continuous for $\alpha \in (0,1]$, and there exist $\rho_{\min},\rho_{\max}$ such that $0 < \rho_{\min} \leq \rho(x) \leq \rho_{\max}$,
        \item \textbf{(S2)} the support of $\mu$ lies in $B(0;R)$, and all the weights are uniformly lower-bounded i.e., $\mu_j \geq \mu_{\min} > 0$ for all $j \in \{1,\ldots,J\}$.
    \end{description}
    Then it holds that, for all $\eps>0$
    \begin{align}\label{eq:bias_main}
    \| T_0^\mu - T_\eps^\mu\|^2_{L^2(\rho)} \leq C_0
    e^{C_1\eps^{\alpha}} \eps\,,
    \end{align}
    where the constants depend on $\rho_{\min}, \rho_{\max}, d, R, \mu_{\min}, J$,  $\min_{i\neq j}\|y_i - y_j\|$, and on the maximum angle formed by three non aligned points among the atoms $\{y_j\}_{j=1}^J$.
\end{prop}
\begin{remark}
Following the asymptotic results of \citet{altschuler2021asymptotics}, we can take a limit of \eqref{eq:biasbound_prelim}, resulting in the computation
\begin{align*}
    \limsup_{\eps\to\infty}\eps^{-1}\|T_0^\mu - T_\eps^\mu\|^2_{L^2(\rho)} &= \sum_{i,j}\frac{\|y_i - y_j\|h_{ij}(0)}{2}\int_0^\infty\Bigl( 1 + \tfrac{\mu_i}{\mu_j}e^{u/2} \Bigr)^{-1} \dd u \\
    &= \sum_{i,j}\|y_i - y_j\|h_{ij}(0) \log(1 + \mu_j/\mu_i)\,,
\end{align*}
where we used that $t\mapsto h_{ij}(t)$ is continuous at $t=0$ (which holds if, for instance, $\rho$ has an upper bounded density with compact support). We conjecture that this quantity is uniformly bounded for all discrete measures.
\end{remark}

Combined with \eqref{eq:bias_variance_decomp} and \cref{cor:stab_1}, we can state and prove our main theorem for this section. While the conditions do not allow for arbitrary discrete measures, we stress that they permit a wide class of discrete measures, and in particular measures supported on different masses.
To our knowledge, this is the first general improvement to the stability bound of \cite{delalande2021quantitative}, even in the semi-discrete case.
\vspace{-5mm}
\begin{theorem}[Near-tight stability in the semi-discrete setting]\label{thm:stability_semidiscrete}
    Suppose $\rho$ satisfies \textbf{(S1)}, and both $\mu$ and $\nu$ each independently satisfy \textbf{(S2)} (with possibly all different parameters). 
    Then
    \begin{align*}
        \|T_0^\mu - T_0^\nu\|_{L^2(\rho)} \lesssim W_2^{1/3}(\mu,\nu)\,,
    \end{align*}
    where the underlying constants depend on those from \cref{prop:bias_main}. 
\end{theorem}
\begin{proof}
First, we note that if $W_2(\mu,\nu)\geq 1$, then $ \| T_0^\mu - T_0^\nu\|_{L^2(\rho)}\leq 2R\leq 2RW_2(\mu,\nu)^{1/3}$, so the only case of interest is when $W_2(\mu,\nu) \leq 1$.
 Then, using the decomposition \eqref{eq:bias_variance_decomp} with \cref{cor:stab_1} and two applications of \eqref{eq:bias_main}, we obtain
\begin{align*}
    \| T_0^\mu - T_0^\nu\|_{L^2(\rho)}^2 &\leq 4\max\{C_0(\mu) e^{C_1(\mu)\eps^\alpha},C_0(\nu)e^{C_1(\nu)\eps^\alpha}\}\eps + (2 + 4R^4/\eps^{2})W_2^2(\mu,\nu)\,.
\end{align*}
Choosing $\eps= W_2^{2/3}(\mu,\nu)\leq 1$, we obtain our desired rate with a prefactor scaling like $C_0e^{C_1}+1$, where we choose the worse constant arising from the bias terms between $\mu$ and $\nu$. 
\end{proof}

\begin{remark}
Closest to this result is that of \cite{bansil2022quantitative}: when $\mu$ and $\nu$ are supported on the \emph{same} atoms, the following bound holds\footnote{This bound is not explicitly written in the paper but it can be extracted from their Theorem 1.3.}
\begin{align}\label{eq:tv_bound}
     \| T_0^{\mu} - T_0^{\nu}\|_{L^2(\rho)}^2 \leq (J-1) \text{diam}(\Omega)^2 \tv{\mu}{\nu}\,,
\end{align}
where $\tv{\cdot}{\cdot}$ is the total variation distance, and $J$ is the number of atoms in the support of $\mu$ and $\nu$.
\Cref{thm:stability_semidiscrete} implies meaningful bounds in some situations in which~\eqref{eq:tv_bound} fails to do so.
For example, consider the simple setting where $\rho = \text{Unif}(B(0,1))$ and $\mu_\theta = \tfrac12(\delta_{e_\theta} + \delta_{-e_\theta})$ with $e_\theta = (\cos(\theta),\sin(\theta))$ for $ 0 < \theta \ll \pi/2$. It is easy to verify that $W_2(\mu_0,\mu_\theta) \asymp \theta$, so \cref{thm:stability_semidiscrete} gives $\| T_0^{\mu} - T_0^{\nu}\|_{L^2(\rho)} \lesssim \theta^{1/3}$. On the other hand, since $\tv{\mu_0}{\mu_\theta} = 1$, the bound~\eqref{eq:tv_bound} is vacuous.
\end{remark}

\section{Conclusion}
In this work, we provide a tight characterization of the stability of entropic Brenier maps under variations of the target measure. Our bounds do not require any of the measures to have densities, and are strengthened when the entropic Brenier potentials are smooth. 
As an application, we used our entropic bounds combined with existing bounds on the bias of semi-discrete entropic Brenier maps to prove improved stability estimates for semi-discrete transport maps when the discrete measures and smooth source measure satisfy certain regularity assumptions.

\section*{Acknowledgements}
AAP thanks Sinho Chewi for helpful discussions. JNW is supported by the Sloan Research Fellowship and NSF grant DMS-2210583. AAP thanks NSF grant DMS-1922658 and Meta AI Research for financial support.

\appendix

\section{Proof of the bias term}\label{sec: biases}
Recall that our target measures are discrete measures of the form
\begin{align*}
    \mu= \sum_{j=1}^J \mu_j \delta_{y_j}\,.
\end{align*}
and that we write the Laguerre cells as $L_i$ for $i\in \{1,\dots,J\}$.

We  require the following definitions, which we borrow from \citet{altschuler2021asymptotics}. For $x \in L_i$ and any other $j \in \{1,\dots,J\}$, we write
\begin{align*}
    \Delta_{ij}(x) \defeq 2(\langle x, y_i - y_j\rangle - \psi_0^\mu(y_i) + \psi_0^\mu(y_j))\,,
\end{align*}
and $H_{ij}(t) = \{x\in L_i:\ \Delta_{ij}(x)=t\}$, which represents the trace on $L_i$ of the hyperplane spanned by the boundary between $L_i$ and $L_j$, shifted by $t$ (should the two cells have non-empty intersection). Moreover, we have the following co-area formula: for every nonnegative measurable function $f:\R\to {\R_+}$,
\begin{equation*}
    \int_{L_i} f(\Delta_{ij}(x)) \rho(x) \dd x = \frac{1}{2\|y_i-y_j\|} \int_0^\infty f(t) h_{ij}(t) \dd t,
\end{equation*}
where 
\begin{align}\label{eq:defhij}
    h_{ij}(t) =\int_{H_{ij}(t)} \rho(x) \dd \cH_{d-1}(x)\,,
\end{align}
and $\cH_{d-1}$ is the $(d-1)$-dimensional Hausdorff measure.

\begin{proof}[Proof of \cref{prop:bias_main}]
Let $i \in \{1,\ldots,J\}$ and let $x \in L_i$. For $j\in \{1,\dots,J\}$ other than $i$, we have the upper bound
\begin{align*}
    \pi_\eps^\mu(y_j|x) &= \mu_i e^{(\langle x,y_j\rangle -\phi_\eps^\mu(x)- \psi_\eps^\mu(y_j))/\eps} \\
    &=\frac{\mu_i e^{(\langle x,y_j\rangle - \psi_\eps^\mu(y_j))/\eps}}{\sum_{k=1}^J \mu_k e^{(\langle x,y_k\rangle - \psi_\eps^\mu(y_k))/\eps} } \\
    &\leq \frac{\mu_j e^{(\langle x,y_j \rangle - \psi_\eps^\mu(y_j))/\eps} }{\mu_i e^{(\langle x,y_i \rangle - \psi_\eps^\mu(y_i))/\eps} + \mu_j e^{(\langle x,y_j \rangle - \psi_\eps^\mu(y_j))/\eps}}\,.
\end{align*}
Adding and subtracting appropriate factors of $\psi_0^\mu(y_i)$ and $\psi_0^\mu(y_j)$, we obtain
\begin{align*}
    \pi_\eps^\mu(y_j|x) &\leq e^{2\|\psi_0^\mu - \psi_\eps^\mu\|_\infty/\eps} \frac{\mu_j e^{(\langle x,y_j \rangle - \psi_0^\mu(y_j))/\eps} }{\mu_i e^{(\langle x,y_i \rangle - \psi_0^\mu(y_i))/\eps} + \mu_j e^{(\langle x,y_j \rangle - \psi_0^\mu(y_j))/\eps}}  \\
    &= e^{2\|\psi^\mu_0 - \psi^\mu_\eps\|_\infty/\eps} \Bigl( 1 + \tfrac{\mu_i}{\mu_j}e^{\Delta_{ij}(x)/2\eps} \Bigr)^{-1}\,,
\end{align*}
By an application of Jensen's inequality, we have
\begin{align*}
    \|T_\eps^\mu(x) - y_i\|^2 &\leq \sum_{j=1}^J \pi_\eps^\mu(y_j|x)\|y_i - y_j\|^2 \\
    &\leq e^{2\|\psi^\mu_0 - \psi^\mu_\eps\|_\infty/\eps} \sum_{j=1}^J \|y_i - y_j\|^2 \Bigl( 1 + \tfrac{\mu_i}{\mu_j}e^{\Delta_{ij}(x)/2\eps} \Bigr)^{-1}\,,
\end{align*}
so integrating against $\rho$ (partitioned into the $J$ Laguerre cells) yields
\begin{align*}
    \|T_\eps^\mu - T_0^\mu\|^2_{L^2(\rho)} &\leq e^{2\|\psi^\mu_0 - \psi^\mu_\eps\|_\infty/\eps} \sum_{i,j} \|y_i - y_j\|^2 \int_{L_i} \Bigl( 1 + \tfrac{\mu_i}{\mu_j}e^{\Delta_{ij}(x)/2\eps} \Bigr)^{-1} \dd \rho(x) \\
    &= e^{2\|\psi^\mu_0 - \psi^\mu_\eps\|_\infty/\eps} \sum_{i,j} \|y_i - y_j\|/2 \int_{0}^\infty h_{ij}(t)\Bigl( 1 + \tfrac{\mu_i}{\mu_j}e^{t/2\eps} \Bigr)^{-1} \dd t  \\
    &= e^{2\|\psi^\mu_0 - \psi^\mu_\eps\|_\infty/\eps}\eps \sum_{i,j} \|y_i - y_j\|/2 \int_{0}^\infty h_{ij}(u\eps)\Bigl( 1 + \tfrac{\mu_i}{\mu_j}e^{u/2} \Bigr)^{-1} \dd u\,,
\end{align*}
where the second line follows from the definition of the co-area formula, and the last line is a change of variables $u = t/\eps$. This gives \eqref{eq:biasbound_prelim}.

With the additional assumptions \textbf{(S1)} and \textbf{(S2)}, we can use Corollary 3.2 of \citet{delalande2021nearly}, which tells us that
\begin{align}
    \eps^{-1}\|\psi_0^\mu - \psi_\eps^\mu\|_\infty \leq C_1 \eps^{\alpha}\,,
\end{align}
where the underlying constant depends on $d, R, J, \mu_{\min}, \min_{i\neq j}\|y_i - y_j\|, \rho_{\min}, \rho_{\max}$, and on the maximum angle formed by three non aligned points among the atoms $\{y_j\}_{j=1}^J$. This gives 
and upper bound of
\begin{align*}
    \|T_\eps^\mu - T_0^\mu\|^2_{L^2(\rho)} \leq e^{C_1\eps^\alpha} \Bigl(\sum_{i,j} \|y_i - y_j\|/2 \int_{0}^\infty h_{ij}(u\eps)\Bigl( 1 + \tfrac{\mu_i}{\mu_j}e^{u/2} \Bigr)^{-1} \dd u)\eps \,.
\end{align*}
Since $\|y_i-y_j\| \leq 2R$, the masses $\mu_i,\mu_j$ are larger than $\mu_{\min}$, and $h_{ij}(u \eps)$ is bounded under our assumptions on $\rho$, the proof is concluded. 
\end{proof}

\bibliography{main}

\end{document}